\documentclass[reqno,11pt,a4paper]{amsart}
\usepackage{amsmath,amssymb,latexsym}
\usepackage{enumerate}
\usepackage[all]{xy}
\usepackage{fancyhdr}
\usepackage{color}

\usepackage[english]{babel}

\newcommand{\Q}{\mbox{$\mathbb Q$}}

\newcommand{\n}{\mbox{$\mathcal N$}}

\newcommand{\lrta}{\longrightarrow}

\newcommand{\A}{\mathcal{A}}

\newcommand{\C}{\mathcal{C}}
\newcommand{\J}{\mathcal{J}}

\newcommand{\K}{K_\infty}

\newcommand{\QQ}{\mathbb{Q}}
\newcommand{\ZZ}{\mathbb{Z}}

\newcommand{\sel}{\mathrm{Sel}}
\newcommand{\Vis}{\mathrm{Vis}}
\newcommand{\gal}{\mathrm{Gal}}

\newcommand{\ilim}{\varinjlim}

\usepackage[OT2,OT1]{fontenc}
\newcommand\cyr{%
\renewcommand\rmdefault{wncyr}%
\renewcommand\sfdefault{wncyss}%
\renewcommand\encodingdefault{OT2}%
\normalfont\selectfont}

\DeclareTextFontCommand{\textcyr}{\cyr}
\newcommand{\Sha}{{\mbox{\textcyr{Sh}}}}

\newtheorem{theorem}{Theorem}[section]

\newtheorem{lemma}[theorem]{Lemma}
\newtheorem{corollary}[theorem]{Corollary}

\newtheorem{assumption}{Assumption}

\newtheorem{example}{Example}
\newtheorem{remark}[theorem]{Remark}

\newcommand{\Keywords}[1]{\par\noindent
{\small{Keywords and phrases}: #1}}

\newcommand{\AMS}[1]{\par\noindent
{\small{AMS Subject Classification}: #1}}

\author{Sudhanshu Shekhar}
\address{Indian Institute of Science Education and Research, Mohali}
\email[]{sshekhars2012@gmail.com}

\begin{document}

\title{Visibility of Shafarevich-Tate group of abelian varieties over number field extensions. }
\begin{abstract}
Given an abelian variety $J$ and an abelian subvariety $A$ of $J$ over a number field $K$, we study the visible elements of the Shafarevich-Tate group of $A$ with respect to $J$ over certain number field extension $M$ of $K$. The notion of visible elements in Shafarevich-Tate group of an abelian variety was introduced by Mazur. 
In this article, we study the image of Visible elements of $A$ with respect to $J$ under the natural restriction map of the Galois cohomology of $A$ over $K$ to the Galois cohomology of $A$ over $M$.   In particular, for a fixed odd prime $p$, we investigate the conditions under which visible elements of order $p$ can be produced over a quadratic extension  or a degree $p$ extension  $M$ of $K$. 
\end{abstract}

\maketitle
\let\thefootnote\relax\footnotetext{
\AMS{14H52}
}

\section{introduction}\label{intro}
Let $L$ be a number field and $\bar{L}$ denote a fixed algebraic closure of $L$. Let $A$ be an abelian variety defined over $L$.  For a number field extension $K/L$,  the Shafarevich-Tate group ${\cyr \Sha}(A/K)$ of $A$ over $K$ is defined by the exact sequence
\begin{equation}\label{sha}
 0 \lrta {\cyr \Sha}(A/K) \lrta H^1(K,A) \stackrel{\prod_w\delta_w}{\lrta} \prod_{w} H^1(K_w,A). 
\end{equation}
Here, $w$ varies over the set of primes of $K$, 
and $\prod\gamma_w$ denotes the natural restriction map of Galois cohomology groups.
The Shafarevich-Tate group ${\cyr \Sha}(A/K)$ of $A$ over $K$ measures the failure of the local to global principal of certain torsors. In his article \cite{Mz}, Mazur introduced the notion of visibility of elements in  ${\cyr \Sha}(A/K)$. Using this one can produce non-trivial elements in ${\cyr \Sha}(A/K)$. Recall that,  for every integer $n>1$, by Kummer theory of abelian varieties, we have an exact sequence
\[   0 \lrta A(K)/n \lrta \sel(A[n]/K)  \lrta   {\cyr \Sha}(A/K)[n] \lrta 0. \]
where $[n]$ denotes the $n$-torsion subgroup, $\sel(A/K)$ denotes the Selmer group of $A$ over $K$ defined by the exact sequence,
\begin{equation}\label{sel}
 0 \lrta \sel(A[n]/K) \lrta H^1(K,A[n]) \stackrel{\prod_w\delta_w}{\lrta} \prod_{w} H^1(K_w,A).
\end{equation}
Here, $A[n]$ is the $n$-torsion subgroup of $A$ over a fixed algebraic closure $\bar{L}$ of $L$.
In particular, the visible elements in ${\cyr \Sha}(A/K)$ helps us in finding much sharper bound on the Mordell-Weil rank of an abelian variety $A$. 

If $i : A  \hookrightarrow J$ is a closed immersion of abelian varieties over $L$ and $K/L$ be a field extension then the subgroup of $H^1(K,A)$ visible with respect to $i$ is
\[ \Vis_J(K,A)= ker (H^1(K,A) \lrta H^1(K,J)) .\]

For an finite extension $K$ of $L$, the visible subgroup of ${\cyr \Sha}(A/K)$ with respect to $i$ is defined as 
\[  \Vis_J({\cyr \Sha}(A/K)) := {\cyr \Sha}(A/K)\cap \Vis_J(K,A) = ker ({\cyr \Sha}(A/K)\lrta {\cyr \Sha}(J/K) )\]

The visible subgroup of the Shafarevich group of an abelian variety has been extensively studied by several mathematicians (see for example \cite{CM}, \cite{AS}, \cite{AS1}).

For  a Galois extension $L$ of $\QQ$ and let  $e_p(L)$ denote the ramification index of $L$ at a prime $p$.  Let $c_{v,A}(L)$ denote the Tamagawa number of an abelian variety $A$ at a finite prime $v$ of $L$.  
We recall  the following important result proved in \cite{AS}.

\begin{theorem}\label{AS}
Let  $A$, $B$ be abelian  sub varieties of an abelian variety $J$  defined over $L$ such that $A\cap B$ is finite and $B[n]\subset A$ as a subgroup scheme of $J$ for  an odd positive integer $n$.  Let $N$ be an integer divisible by the primes of bad reduction of $A$ or  $B$ over $L$. Suppose that  $n$ is coprime to $N$, $B(L)[n]=0$ and $(J/B)(L)[n]=0$. Then we have a natural map  
\[  B(L)/nB(L)  \stackrel{\phi_L}{\longrightarrow}  \Vis_J(A/L)  \]
such that the cardinality of the kernel of $\phi_K$ is bounded by $n^r$ where $r$ is the Mordell-Weil rank of $A$. 
Further, if $e_p(L)< p-1$ for all primes $p|n$ and that  $n$ is coprime to $c_{v,A}(L)$ and $c_{v,B}(L)$ for every prime $v$ of $L$ then $\phi_L$ factors through
\[  B(L)/nB(L)  \stackrel{\phi_L}{\longrightarrow}  \Vis_J({\cyr \Sha}(A/L))  \]
\end{theorem}

In particular, if the Mordell-Weil rank of $A$ is less than the Mordell-Weil rank of $B$ then from the above theorem, we can construct non-trivial elements of the Shafaravich-Tate group of $A$ over $K$. 
Aim of this article is to  study certain visible elements in the Shafarevich Tate group of an abelian variety over number field extensions of $L$  via a relative version of the above theorem. 
Let $K$ be a Galois extension of $L$. For a Galois extension $M/L$ containing $K$, let $\tau_{K}^{M}$ denote the natural restriction map
\[  \Vis_J(K,A)\lrta \Vis_J(M,A)  \]  
induced by the natural restriction map
\[  \tau_K^M : H^1(K,A) \lrta H^1(M,A) .\]

If the Tamagawa number at a prime $v$ is not coprime to $n$ then the map $\phi_K$ may not factor through $\Vis_J({\cyr \Sha}(A/K)$. Nevertheless, the composition of maps $\tau_{K}^{M}\circ \phi_K$ may factor through $\Vis_J({\cyr \Sha}(A/M))$. In this article we study this phenomenon and provide several illustrative examples where the composition $\tau_{K}^{M}\circ \phi_K$ is a non trivial map. 
We prove our result in two mutually exclusive cases. In section \ref{prime}, we consider the case when $n$ is coprime to the degree of $M/K$. 

Let $L/\QQ$ be a Galois extension and fix an odd integer $n$. Let $e_p=e_p(L)$ be the ramification index of $L$  at prime number $p$.  For an abelian variety $A$ defined over a number field $K$, let $N_A(K)$ be the integer defined as the product of  residue characteristics at the primes of bad reductions of $A$ over $K$. 
Let $K$ and $M$  be a finite Galois extensions of $\QQ$ such that $L\subset K\subset M$. 
For a prime $v$ of $K$, let $c_{v,A}(K)$ be the Tamagawa number of $A$ at $v$. Let $A$ and $B$ be abelian subvarieties of an abelian variety $J$ defined over $L$.   
 For a  character  $\chi$ of $\gal(M/K)$ and an abelian variety $A$, let  $A^\chi$ denotes the twist of $A$ by $\chi$.  For the definition and a detailed discussion on the twist of an abelian variety by a finite order character of Galois group, we refer the reader to \cite{Ki}.   Then we have the following theorem. 

\begin{theorem}
Let  $A$ and $B$ are abelian subvarieties of an abelian variety $J$ defined over $L$ such that $A\cap B$ is finite and $B[n] \subset A\cap B$. Further assume that \\
(i) $B(M)[n]=0$ and $(J/B)(M)[n]=0$. \\
(ii) $n$ is coprime  to $\prod_v c_{v,A^\chi}(K)c_{v,B^\chi}(K)$ where $v$ varies over the primes of bad reduction of $A^\chi$ or $B^\chi$.\\
(iii) The order of $\gal(M/K)$ is prime to $n$. \\
(iv)  $gcd(n,N)=1$ where $N(L)=lcm(N_A(L),N_B(L))$.\\  
(v) The ramification index $e_p(L)$ of $L$ at primes $p|n$ satisfies the relation  $e_p(L)<p-1$.\\
Then there exists a natural map 
\[  B{^\chi} (K)/nB^{\chi}(K)  \stackrel{\tau^M_K \phi_K}{\longrightarrow}  \Vis_J({\cyr \Sha}(A/M))  \]
such that the  cardinality of  the kernel of $\tau^M_K\phi_K$ is bounded by $n^r$ where $r$ denotes  the rank of $A^\chi(K)$.  In particular, $\tau_K^M\circ \phi_K$ is injective if  the rank of $A^\chi(K)$ is zero. 
\end{theorem}
Note that  in the above theorem, $A^\chi$ and $B^\chi$ may not have good reduction at primes dividing $n$. We also mention that to prove our result, we do not need to assume that the ramification index $e_v(K)$ of $K$ (resp. $e_p(M)$ of $M$)  at  primes dividing $n$ satisfy the assumption that $e_p(K)< p-1$ (resp.  $e_p(M)<p-1$).   Instead, we assume that $A$, $B$ and $J$ are defined over $L$ and the ramification index $e_p(L)$ of $L$ at primes dividing $n$ satisfy  that $e_p(L)< p-1$.  We also provide examples where $\tau_K^M\phi_K$ is non-trivial.

In the next  section \ref{tam}, we study certain conditions on the conductor of elliptic curves, under which the Tamagawa numbers of elliptic curves are $p$-adic unit for a fixed odd prime $p$ and  as a consequence  of this we prove the following

\begin{theorem}
Let $A$ and $B$ be elliptic curves defined over $L$ with good reduction at primes dividing $p$  of $K$ such that $A[p]\cong B[p]$ as  $\gal(\bar{L}/L)$-modules. Further suppose that \\
(a)  $e_p=e_p(L)<p-1$.\\
(b)  the prime to $p$-conductor of $A[p]$ is same as the conductor of $A$ and $B$ respectively over $K$.\\
(c) $A[p]$ is an irreducible $\gal(\bar{K}/K)$-module.\\
If the Mordell-Weil rank of $B$ is greater then the Mordell-Weil rank of $A$ over $K$ then ${\cyr \Sha}(A/K)$ contains an element of order $p$.
\end{theorem}
For the definition of the conductor  of the $p$-adic Galois representation associated to an elliptic curve $A$ and the residual Galois representation $A[p]$ associated to $A$ at $p$, we refer the reader to \cite[Definition 3.2.27]{Wi}. 
The above theorem can be considered as a generalisation of  \cite[Theorem 1.3]{A} where author considers modular abelian varieties $A$ and $B$  of prime conductors defined over $K=\QQ$. We mention that our method to prove the above theorem can easily be extended to modular abelian varieties also.  For more details, see Theorem \ref{nontrivial} and other related results (for example Theorem \ref{nontrivial1}). Using these results and also results from Iwasawa Theory we study the ``algebraic part" of the leading term of the Hasse-Weil $L$-functions of elliptic curves defined over $\QQ$. We show under certain appropriate set of conditions on elliptic curves of rank $\leq 1$ that the algebraic part of  the leading terms of Hasse-Weil $L$-functions is divisible by $p$ (for more details see Corollary \ref{analy} and \ref{analy1}).  

In this section, we particularly consider the case  when $K=\QQ$ and $M/\QQ$ is a quadratic extension. Given an elliptic curve $E$  defined over $\QQ$ and a prime $p$, it is an important question to know the smallest degree extension over which the Shafarevich-Tate group of $E$ has an element of order $p$ (see for example \cite{C}). In fact, it is widely expected that in general  given an elliptic curve $E$ defined over $\QQ$ and a prime $p$, there exists a quadratic extension $M$ of $\QQ$ such that there is an element of order $p$ in ${\cyr \Sha}(E/M)$. We provide several examples to explain how results in this article may be useful   in finding such quadratic extensions. 

In  section \ref{prime1}, we consider the case when $n=p$ and $M/K$ is an extension of degree divisible by $p$.  In particular, we have the following theorem. 

\begin{theorem}
Let  $A$ and $B$ are abelian subvarieties of an abelian variety $J$ defined over $L$ such that $A\cap B$ is finite and $B[n] \subset A\cap B$. Further assume that \\
(i)  $B(K)[p]=0$ and $(J/B)(K)[p]=0$. \\
(ii) the ramification index  $e_v(M)$ of  $M$ at primes dividing $N(K)/N_A(K)$ is divisible by $p$,\\
(iii) $p\nmid \prod_{v|N_A(K)} c_{v,A}(K)c_{v,B}(K).$\\ 
(iv) $gcd(p,N(L))=1$ where $N(L)=lcm(N_A(L),N_B(L))$. \\
(v)  The ramification index $e_p(L)$ of $L$ at primes $p|n$ satisfies the relation  $e_p(L)<p-1$.\\
 Then there exists a natural map from 
 \[  B(K)/pB(K)  \stackrel{\tau^M_K \phi_K}{\longrightarrow} Vis_J({\cyr \Sha}(M/A))\] 
such that $\tau^M_K \phi_K$ is an injective homomorphism if the rank of $A(M)$ is zero. More generally, if $M$ is a cyclic extension of $K$ of degree $p$  then the $\ZZ/p\ZZ$-rank of the kernel of $\tau^M_K \phi_K$ is bounded by $(rank(A(K))+\frac{rank(A(M))-rank(A(K))}{p-1})$. \\
\end{theorem}
Note that in the above theorem, we do not put any condition on the Tamagawa numbers of $B$ at the primes of good reduction of $A$. 
Finally, we also  prove a generalisation  of the above theorem  when $K$ is replaced by a certain  $p$-adic Lie extension of a number field containing $L$ and discuss an example to illustrate our result. 

\section{visibility over number fields}\label{lemma}
We begin by stating the following two crucial lemmas. We shall use  these lemmas to prove our results without using the conditions in Theorem \ref{AS}  on the ramification index of $K$ at primes dividing $n$.
\begin{lemma}\label{smooth}(\cite[Lemma 3.6]{AS})
Let $\mathcal{I}\stackrel{\phi}{\lrta} \mathcal{C}$ be a smooth surjective morphism of schemes over a strictly Henselian local ring $R$. Then the induced map $\mathcal{I}(R) \lrta \mathcal{C}(R)$ is surjective.
\end{lemma}
\begin{lemma}\label{smooth1}
Suppose that 
\begin{equation}\label{exact1}
0 \lrta A \lrta  J \stackrel{\psi}{\lrta} C \lrta 0 
\end{equation}
is an exact sequence of abelian varieties with good reduction over a finite Galois extension $F$ of $\QQ_p$. Further assume that the ramification index $e_p$ of $F$ at  $p$ satisfies $e_p < p-1$. Then for every finite Galois extension $K$ of $F$, we have the following natural exact sequence 
 \begin{equation}\label{exact2}
 0 \lrta A(K^{ur}) \lrta  J(K^{ur})  \lrta C(K^{ur}) \lrta 0 
 \end{equation}
induced by the exact sequence \eqref{exact1}. Here $K^{ur}$ denotes the maximal unramified extension of $K$. 
\end{lemma}
\begin{proof}
For an extension $K$ of $F$, let $\mathcal{O}_{K^{ur}}$ denote the ring of integers of $K^{ur}$ and put $X_{K^{ur}}:=spec(\mathcal{O}_{K^{ur}})$.  Let $\J$ (resp. $\C$) denote the N\'eron model of  $J$ (resp $C$) over $X_{F^{ur}}$. Since $e_p<p-1$, the natural morphism of schemes  $\J \stackrel{\psi}{\lrta} \C$ induced by $\psi$ is smooth and surjective (see the proof of \cite[Theorem 3.1, case 3]{AS}). The property of smoothness and surjectivity are stable under base change and therefore  the induced map
\[  \J\times_{X_{F^{ur}}} X_{K^{ur}} \lrta  \C\times_{X_{F^{ur}}} X_{K^{ur}} \]
is smooth and surjective.  
Since $J$ and $C$ has good reduction over $F$, we get that the N\'eron model of $J$ (resp $C$ ) over $K^{ur}$ is naturally isomorphic to  $\J\times_{X_{F^{ur}}} X_{K^{ur}}$ (resp. $ \C\times_{X_{F^{ur}}} X_{K^{ur}} $) (see \cite[Chapter 7.4, Corollary 4]{BLR}). Now from the universal property of N\'eron model of $J$ (resp. $C$) over $K^{ur}$  and Lemma \ref{smooth}, we get that the induced map
 $J(K^{ur})  \lrta C(K^{ur})$ is surjective. This implies that the sequence
 \[  0 \lrta A(K^{ur}) \lrta  J(K^{ur})  \lrta C(K^{ur}) \lrta 0 \]
 is exact.
\end{proof}

To prove our results we need the following key lemma. 
\begin{lemma}\label{vis}(\cite[Lemma 3.7]{AS})
Let $A$ and $B$ be abelian subvarieties of an abelian variety $J$ defined over $L$ such that $A\cap B$ is finite and $B[n] \subset A\cap B$. 
Let $K$ be a finite extension of $L$ such that 
$(J/B)(K)[n]=0$ and $B(K)[n]=0$. Then 
there is a natural map
\[   \phi_K  :  B(K) / nB(K)  \lrta \Vis_J(K,A)  \]
such that $ker(\phi) \subset J(K)/(B(K) + A(K) ) $. If the rank of $A(K)$ is zero then $\phi$ is an injective map. More generally,  the cardinality of the kernel of $\phi_K$ is bounded by $n^r$ where $r$ is the Mordell-Weil rank of $A$. 
\end{lemma}
\qed


We shall use results mentioned in this section to produce visible elements over certain extension of number fields. 
We end this section with the following theorem which is an extension of \cite[Theorem 3.1]{AS} over general number fields. The proof of the theorem is similar to \cite{AS}, except that in this case we use Lemma \ref{smooth1} to  state it over number fields $K/L$ which do not satisfy the ramification condition $e_p(K)<p-1$ for primes $p|n$. 

\begin{theorem}\label{improv}
Let  $A$, $B$ be abelian  sub varieties of an abelian variety $J$  defined over a number field $L$ such that $A\cap B$ is finite and $B[n]\subset A$ as a subgroup scheme of $J$.  Let $N(L)$ be an integer divisible by the primes of bad reduction of $A$ or  $B$ over $L$. Let $K$ be a finite Galois extension of $L$. Suppose that  $n$ is coprime to $N(L)$, $B(K)[n]=0$ and $(J/B)(K)[n]=0$. Then we have a natural map  
\[  B(K)/nB(K)  \stackrel{\phi_K}{\longrightarrow}  \Vis_J(A/K)  \]
such that the cardinality of the kernel of $\phi_K$ is bounded by $n^r$ where $r$ is the Mordell-Weil rank of $A$. 
Further, if $e_p(L)< p-1$ for all primes $p|n$ and that  $n$ is coprime to $c_{v,A}(K)$ and $c_{v,B}(K)$ for every prime $v$ of $K$ then $\phi_K$ factors through
\[  B(K)/nB(K)  \stackrel{\phi_K}{\longrightarrow}  \Vis_J({\cyr \Sha}(A/K))  \]
\end{theorem}
\begin{proof}
 Using Lemma \ref{vis} we get a natural map 
\[   \phi  :  B(K) / nB(K)  \stackrel{\phi_K}{\lrta} \Vis_J(K,A)  \]
such that the cardinality of the kernel of $\phi_K$ is bounded by   $n^r$. 
For a prime $v$ of $K$, let $ \kappa_v$ be the natural restriction map 
\[  H^1(K,A) \lrta H^1(K_{w},A)  .\]
Let $x\in B(K)$. 
Using the assumption that $c_{v,A}(K)$ and $c_{v,B}(K)$ are coprime to $n$ for every finite prime $v$ of $K$, we get by an argument similar to  \cite[Theorem 3.1]{AS} that $\phi_K(x)$ belongs to  the kernel of  $\kappa_v$ for ever  prime $v\nmid n$. If $v|n$ then Lemma \ref{smooth1} implies that the sequence
 \[  0 \lrta A(K^{ur}) \lrta  J(K^{ur})  \lrta C(K^{ur}) \lrta 0 \]
 is exact where  $C=J/A$. Now again it follows by an argument similar to \cite[Theorem 3.1]{AS} that $\phi_K(x)$ belong to kernel of $\kappa_v$. This show that  $\phi_K$ factors through $B(K)/n(B(K) \lrta  \Vis_J({\cyr \Sha}(A/K)$. 
\end{proof}

\section{twist of abelian varieties and visible elements}\label{prime}
As before we fix a finite Galois extension $L/\QQ$ and an odd integer $n$. Let $e_p=e_p(L)$ be the ramification index of $L$  at prime number $p$.  Let $K$ be a finite Galois extension of $\QQ$ such that $L\subset K$. We state the following  assumptions which we shall refer as and when needed. 

\begin{assumption}\label{assumption}
(a) Let  $A$ and $B$ are abelian subvarieties of an abelian variety $J$ defined over $L$ such that $A\cap B$ is finite and $B[n] \subset A\cap B$.\\
(b) The ramification index $e_p$ of $L$ at primes $p|n$ satisfies the relation  $e_p<p-1$.\\
(c) $gcd(n,N(L))=1$ where $N(L)=lcm(N_A(L),N_B(L))$. \\
(d) $B(K)[n]=0$ and $(J/B)(K)[n]=0$. 
\end{assumption}



Let $M$ be a finite Galois extension of $K$ and let $\chi$ be a character of $\gal(M/K)$. For an abelian variety $A$, we denote by $A^\chi$ the twist of $A$ by $\chi$. For the definition and a detailed discussion on the twist of an abelian variety by a finite order character of Galois group, we refer the reader to \cite{Ki}.  Now, let $A$ and $B$ be abelian subvarieties of an abelian variety $J$ defined over $L$ as in Assumption \ref{assumption}. Then $A^\chi$ and $B^\chi$ are abelian sub varieties of $J^\chi$ defined over $K$ such that $A^\chi\cap B^\chi$ is finite and $B^\chi[p]\subset A^\chi\cap B^\chi$. Since $A^\chi \cong A$ (resp.  $J^\chi \cong J$) over $M$, we have a natural isomorphism from $Vis_{J^\chi}(M,A^\chi) \lrta Vis_{J}(M,A)$.  Note that $A^\chi$ and $B^\chi$ may not satisfy  Assumption \ref{assumption}(c).  Nevertheless, we have the following theorem. 

\begin{theorem}\label{quadratic}
Suppose that $A$ and $B$ are abelian subvarieties of an abelian variety $J$ defined over $L$ satisfying Assumption \ref{assumption}. Further assume that \\
(i) $B(M)[n]=0$ and $(J/B)(M)[n]=0$. \\
(ii) $n$ is coprime  to $\prod_v c_{v,A^\chi}(K)c_{v,B^\chi}(K)$ where $v$ varies over the primes of bad reduction of $A^\chi$ or $B^\chi$.\\
(iii) The order of $\gal(M/K)$ is prime to $n$. \\
Then the natural map
\[  B^{\chi}(K)/nB^{\chi}(K) \stackrel{\phi_K}{\lrta} \Vis_{J^\chi}(K,A^\chi) \stackrel{\tau_K^M}{\lrta}  \Vis_{J^\chi}(M,A^\chi)  \cong \Vis_J(M,A) \]
factors through a map 
\[  B{^\chi} (K)/nB^{\chi}(K)  \stackrel{\tau^M_K \phi_K}{\longrightarrow}  \Vis_J({\cyr \Sha}(A/M))  .\]
The cardinality of  the kernel of $\tau^M_K\phi_K$ is bounded by $n^r$ where $r$ denotes  the rank of $A^\chi(K)$.  In particular, $\tau_K^M\circ \phi_K$ is injective if the rank of $A^\chi(K)$ is zero. 
\end{theorem}
\begin{proof}
We shall write $\tau$ to denote $\tau^M_K$ and $\phi$ to denote $\phi_K$ for simplicity.  Since $B(M)[n]=0$ and $B^\chi\cong B$ over $M$,  $B^\chi(M)[n]=0$. Therefore $B^\chi(K)[n]=0$. Similarly, $(J^\chi/A^\chi)(K)[n]=0$.  From \ref{vis} we have a natural map
\[    \phi_K  :  B^\chi(K)/nB^\chi(K) \lrta \Vis_{J^\chi}(K,A^\chi)  . \]
 Let $x\in B^\chi(K)$.
For a prime $v$ of $K$, let $\kappa_v(K)$ denotes the restriction map
\[  H^1(K,A^\chi)  \lrta H^1(K_v,A^\chi),  \]
and for prime $w$ of $M$, let $\kappa_w(M)$ denotes the corresponding restriction map for $M$. 
For a prime $w|v$, let $\tau_w$ be the restriction map $H^1(K_v,A) \lrta H^1(M_w,A)$. 
Let $v\nmid n$ be a finite prime of $K$. Since $n$ is coprime to $\prod_v c_{v,A^\chi}(K)c_{v,B^\chi}(K)$, from \cite[Theorem 3.1]{AS} we have that $\kappa_v(K)\phi(x)=0$ which implies that $\tau_w\kappa_v\phi(x)=0$. Therefore $\kappa_w(M)\tau\phi(x)=0$. Thus, the image of $\tau\phi$ is contained in the kernel of $\kappa_w$ for every prime $w\nmid n$ of $M$. 

If $v$ is an infinite prime of $K$ then the order of $\kappa_v(M)\phi(x)$ divides $2$. On the other hand $\phi(x)$ is killed by $n$. Now using the assumption that $n$ is an odd number, we get that $\kappa_v(M)\phi(x)=0$. Once again this implies that  $\tau\phi(x)$ is contained in the kernel of $\kappa_w(M)$ for every infinite prime $w$. Finally, we consider a prime $v|n$ of $K$. Since $A^\chi$ nor $B^\chi$ may be defined over $L$, we cannot use the method of \cite[Theorem 3.1]{AS} to say that  $\kappa_v(K)\phi(x)=0$. Instead, we notice that the map $\tau\phi$ factors through the natural map
\[   B^\chi(K)/nB^\chi(K) \stackrel{\psi}{\lrta} B^\chi(M)/nB^\chi(M) \stackrel{\phi_M}{\lrta} \Vis_{J^\chi}(M,A^\chi)   \] 
where the first map $\psi$ is induced by the natural inclusion map $B(K)\lrta B(M)$. 
Also $A$, $B$ and $J$ are naturally isomorphic to $A^\chi$, $B^\chi$ and $J^\chi$ respectively. Therefore, it is enough to show that $\phi_M\psi(x)$ is in the kernel of $\kappa_w$ for every prime $w|n$. We shall show that for every $y\in B(M)$ and primes $w|n$, $\kappa_w\phi_M(y)=0$. This will imply that the image of $\phi_M\psi$ is contained in the kernel of $\kappa_w$. 

Let $C$ be the abelian variety  $J/A$ defined over $L$. Since ramification index $e_p$ for every prime of $L$ dividing $n$ satisfies $e_p<p-1$,  $A$ and $B$ are defined over $L$, from Lemma \ref{smooth1} we get that the sequence 
\[  0 \lrta  A(M_w^{ur}) \lrta J(M_w^{ur}) \lrta C(M_w^{ur}) \lrta 0  \]
is exact. Now it follows from an argument similar to the proof of \cite[Theorem 3.1]{AS} that the image of an element  $y$ of $B(M)$ in $H^1(M_w,A)$ is zero. Taking $y=\psi(x)$, we get that the image of $x$ in $H^1(M_w,A)$ is zero. Thus we have shown that the image of $x$ is contained in the kernel of $\kappa_w$ for every prime $w$ of $M$. This proves the first part of the theorem. 

From snake lemma we get the exact sequence,
\[  0 \lrta ker(\phi) \lrta ker(\tau\phi) \lrta ker(\tau)  . \]
The kernel of $\tau$ is a subgroup of $H^1(M/K,A(M)$. Since the order of  $\gal(M/K)$ is  coprime to $n$,   $H^1(M/K,A(M)[n]=0$. Therefore $ker(\tau)[n]=0$. This implies that the cardinality of the kernel of  $\phi$ is same as the cardinality of the kernel of $\tau\phi$. Now it follows from \cite[Lemma 3.7]{AS} that the cardinality of  the kernel of $\phi$ is bounded by $n^r$. This completes the proof of theorem.  
\end{proof}

\begin{remark}\label{const}
Let $A$ and $B$ be abelian varieties  defined over a number field $L$ such that $A[p]\cong B[p]$ as $\gal(\bar{L}/L)$-modules for a prime number $p$.  Put $J=(A\oplus B)/A[p]$ where we view  $A[p]$ as a subvariety of $A\oplus B$ via the diagonal embedding.   We view $A$ and $B$ as an abelian sub variety of $J$ via the embedding 
\[  A \hookrightarrow A\oplus B  \twoheadrightarrow (A\oplus B)/A[p]  \]  
and 
\[  B \hookrightarrow A\oplus B  \twoheadrightarrow (A\oplus B)/A[p]  \]  
 respectively.  Note that under this map $A[p]$ maps to $A\cap B=A[p]=B[p]$ in $J$. 
Now, suppose that for an extension $K$ of $L$, $A[p]$ is an irreducible $\gal(\bar{K}/K)$-module. This in particular implies that $A(K)[p]=0$. Further, the embedding $B \hookrightarrow J$ induces an isogeny from $B\lrta J/A$. Since $A[p]$ is an irreducible $\gal(\bar{K}/K)$-module and $B$ is isogenous to $J/A$, we get that $(J/A)[p]$ is also an irreducible $\gal(\bar{K}/K)$-module. Therefore $(J/A)(K)[p]=0$. In particular, we see that  $A$, $B$ and $J$  satisfies Assumption \ref{assumption}(a) and (d) over $K$ for $n=p$.  
\end{remark}

Before we discuss the implications of the above theorem in next section, consider the following example.\\
\begin{example}\label{ex1}
{ \normalfont
Let $E_1$ and $E_2$ be the familiar pair of elliptic curves from \cite{GV} defined by the equations
\[  E_1 :  y^2 = x^3 +x - 10 \]
and 
\[  E_2 : y^2 =  x^3 - 584 x + 5444 \]
respectively over $\QQ$. The conductor of $E_1$ over $\QQ$ is $52=4\times 13$ and the conductor of $E_2$ over $\QQ$ is $364=4\times 7 \times 13$. We have that $E_1[p]\cong E_2[p]$ as $\gal(\bar{\QQ}/\QQ)$-module for $p=5$.  
The Tamagawa number $c_{v,E_1}$ and $c_{v,E_2}$ of $E_1$ and $E_2$ respectively are $5$-adic unit for every prime $v|52$. But the Tamagawa number of $E_2$ at $7$ is $5$.
Let $\chi$ be the quadratic character associated to the quadratic extension $\QQ(\sqrt{59})/\QQ$.
The Weierstrass equation for $E_1^\chi$ is given by 
\[  E_1^\chi :  y^2 =  x^3 -  3481 x + 2053790  \]
and the Weierstrass equation for $E_2^\chi$ is given by
\[   E_2^\chi  :  y^2 =  x^3 - 2032904 x + 1118083276 .  \]
The conductor of $E_1^\chi$ is equal to $724048= 2^2\times 13\times 59^2$ and conductor of $E_2^\chi$ is equal to $5068336 = 2^2\times 13\times 7 \times 59^2$.  The Tamagawa number of $E_1^\chi$ (resp. $E_2^\chi$) are $5$-adic units at all primes of bad reduction.  Put $J=(E_1\oplus E_2)/E[5]$ where we view  $E[5]$ as a subvariety of $E_1\oplus E_2$ via the diagonal embedding.   
 It can be checked that $E_1[5]$ is an irreducible $\gal(\bar{M}/M)$-module. From Remark \ref{const}, we get that $E_1$, $E_2$ and $J$ satisfy all the assumptions of Theorem \ref{quadratic}.  Thus we get a map 
\[  \tau^M_K\phi_K : E_2^\chi(\QQ) \lrta  \Vis_J({\cyr \Sha}(E_1/\QQ(\sqrt{59})))    \]
as in Theorem \ref{quadratic}.  Using Sage we get that the rank of $E_1^\chi$ is zero and the rank of $E_2^\chi$ is two. In particular, Theorem \ref{quadratic} implies that the image of $\tau^M_K\phi_K$ is a visible subgroup of ${\cyr \Sha}(E_1/\QQ(\sqrt{59})$ of order $25$. 
}
\end{example}
\qed

\section{Tamagawa numbers and visibility}\label{tam}
As in the previous section, we consider a number field $L$ and a Galois extension $K/L$. Fix an odd prime $p$.  
Next we state a lemma on vanishing of $p$-part of the Tamagawa number of an elliptic curve $E$  at a prime $v\nmid p$ of $K$ satisfying the assumption that the prime to $p$ conductor of $E[p]$ as $\gal(\bar{K}_v/K_v)$-module is same as the conductor of  of $E$ at $v$. Recall that the conductor of $E$  at a prime $v$ is same as the conductor of the $p$-adic Galois representation associated to $E$ for $v\nmid p$ which is in fact independent of the choice of prime $p$. For the definition of the conductor of  Galois module $E[p]$ and the $p$-adic Galois representation associated to an elliptic curve, we refer the reader to \cite[Definition 3.2.27]{Wi}.

 \begin{lemma}\label{compare}
 Let $E$  be an elliptic curve defined over a number field $K$ and  $v\nmid p$ be a finite prime of $K$. Suppose  that the conductor of $E[p]$ as a module over the Galois group of $\gal(\bar{K}_v/K_v)$ is same as the conductor of  $E$ at $v$. Then the Tamagawa number $c_v(E)$ over $K$ is a $p$-adic unit.
   \end{lemma}
 \begin{proof}
 Let $I_v$ denote the inertia subgroup of $\gal(\bar{K}_v/K_v)$ and  $E(p)$  denote the $p$-primary Galois submodule of $E(\bar{K})$. From the proof of \cite[Lemma 5.4]{MR} we have that $H^0(F,E)/p=H^0(F,E(p))/p$ for every every finite extension $F$ of $K_v$. By taking direct limit over the unramified extensions of $K_v$, we get that $H^0(I_v,E)/p=H^0(I_v,E(p))/p$.   Since the conductor of $E[p]$ as a module over $\gal(\bar{K}_v/K_v)$ is equal to the conductor of  $E$, it follows from the proof of   \cite[Lemma 4.1.2] {EPW}  that $H^0(I_v,E(p))$ is a $p$-divisible abelian group.  We mention that  in \cite[Lemma 4.1.2] {EPW}  the authors consider the case when $K=\QQ$. But a similar argument works for a general number field $K$. In fact, by an argument similar to \cite[Lemma 4.1.2] {EPW},  it can be shown that 
 \[  dim_{\ZZ/p}(E[p])^{I_v} = dim_{\QQ_p}(V_pE)^{I_v}  \]
 where $V_pE$ denotes the $p$-adic Galois representation associated to $E$. This in particular implies that $E(p)^{I_v}$ is $p$-divisible. 
 This implies that $H^0(I_v,E(p))/p=H^0(I_v,E)/p=0$. We consider the exact sequence
 \[   0 \lrta H^0(I_v,E[p]) \lrta H^0(I_v,E) \stackrel{p}{\lrta} H^0(I_v,E) \lrta 0. \]
The second map in the exact sequence is given by multiplication by $p$. Note that here we have used the fact that $H^0(I_v,E)[p]=H^0(I_v,E[p])$.   From  the associated long exact cohomology sequence, we have the exact sequence
\begin{align*}
  H^1(K^{ur}_v/K_v,E^{I_v}[p]) \lrta H^1(K^{ur}_v/K_v, & E^{I_v})  \stackrel{p}{\lrta} H^1(K^{ur}_v/K_v,E^{I_v})  \\
 &  \lrta H^2(K^{ur}_v/K_v,E^{I_v}[p]) .
 \end{align*}
Since $\gal(K^{ur}_v/K_v)$ has $p$-cohomological dimension $1$, $H^2(K^{ur}_v/K_v,E^{I_v}[p])=0$. This implies that $H^1(K^{ur}_v/K_v,E^{I_v})$ is $p$-divisible. But we also know that $H^1(K^{ur}_v/K_v,E^{I_v})$ is a finite group of cardinality $c_v(E)$. Therefore  $c_v(E)$ is a $p$-adic unit.  
 \end{proof}

Throughout this section,  we shall  assume that  $e_p=e_p(L)<p-1$ (see Assumption \ref{assumption}). We also  fix  a pair of elliptic curves $A$ and $B$ defined over $L$ such that $A[p]\cong B[p]$ as $\gal(\bar{L}/L)$-modules.   Put $J=(A\oplus B)/A[p]$, where we view $A[p]$ as a subgroup scheme of $A\oplus B$ via the diagonal embedding. Throughout this section we shall assume that $A[p]$ is an irreducible $\gal(\bar{K}/K)$-module.  In particular this assumption implies that $A(K)[p]=(J/A)(K)[p]=0$.
\begin{theorem}\label{nontrivial}
Suppose that the prime to $p$-conductor of $A[p]$ is same as the conductor of $A$ and $B$ respectively over $K$.   If the Mordell-Weil rank of $B$ is greater then the Mordell-Weil rank of $A$ over $K$ then ${\cyr \Sha}(A/K)$ contains an element of order $p$.
\end{theorem}
\begin{proof}
We consider the abelian variety $J=(A\oplus B)/A[p]$ where $A[p]$ is considered as the subgroup scheme of $A\oplus B$ via the diagonal embedding. From Lemma \ref{compare} we have that $c_{v,A}(K)$ and $c_{v,B}(K)$ are $p$-adic units for every finite prime $v$ of  $K$.   Thus  the theorem follows from Theorem \ref{improv}.
\end{proof}

For an abelian variety $A$ over $K$, let $\bar{\n}(K)$ denote the prime to $p$ conductor of the $\gal(\bar{K}/K)$-module $A[p]$ over $K$ and  $\n_A(K)$ denotes the conductor  $A$ over $K$. Note that if $A$ is a semi-stable elliptic curve over $K$ then $N_A(K)=\n_A(K)$.  For an abelian variety $A$ defined over $\QQ$,  we shall denote $\bar{\n}_A(\QQ)$ by $\bar{\n}_A$ and $\n_A(\QQ)$ by $\n_A$. Also we shall denote $N_A(\QQ)$ by $N_A$. Then more generally we have,

\begin{theorem}\label{nontrivial1}
Suppose that $A$ and $B$ elliptic curves defined over  $L$ with good reduction at primes dividing $p$  such that $A[p]\cong B[p]$ as $\gal(\bar{L}/L)$-modules.  Further we assume   that \\
(a) if $v$ is prime of $K$ dividing $\n_A(K)/\bar{\n}(K)$ (resp. $\n_B(K)/\bar{\n}(K)$) then $A$ (resp. $B$) has non-split multiplicative reduction at $v$. \\
(b)  $A$ and $B$ are semi-stable over $K$ .\\
(c) $A[p]$ is an irreducible $\gal(\bar{K}/K)$-module. \\
(d)  The Mordell-Weil rank of $B$ is greater then the Mordell-Weil rank of $A$ over $K$. \\
 Then ${\cyr \Sha}(E/K)$ contains an element of order $p$.
\end{theorem}
\begin{proof}
Let $v$ be a prime of $K$ dividing $\bar{\n}(K)$. Since $A$ is a semi-stable elliptic curve, $\n_A(K)$  (resp. $\n_B(K)$) are square free and therefore the conductor of $A$ (resp. $B$) at  $v$  over $K$ is equal to the conductor of $A[p]$ at $v$. Therefore from Lemma \ref{compare}  the Tamagawa number $c_{v,A}(K)$ (resp. $c_{v,B}(K)$ ) of $A$ (resp. $B$) over $K$ is a $p$-adic unit. Now we consider the case when $v|\n_A(K)/\bar{\n}(K)$.  Since $A$ has non-split multiplicative reduction at $v$, again we get that $c_{v,A}(K)$ is a $p$-adic unit. Similarly if $v|\n_B(K)/\bar{\n}(K)$ then $c_{v,B}(K)$ is a $p$-adic unit. Therefore the Tamagawa numbers of $A$ and $B$ are $p$-adic units over $K$. Now the theorem follows by an argument similar to Theorem \ref{nontrivial}. 
\end{proof}

In this context we mention that Theorem \ref{nontrivial} can be considered as a generalisation of  \cite[Theorem 1.3]{A} where author considers modular abelian varieties $A$ and $B$  of prime conductors defined over $K=\QQ$. Note that if $A$ and $B$ are elliptic curves,  $A[p]\cong B[p]$ is an irreducible $\gal(\bar{K}/K)$-module and the conductor of $A$ and $B$ are prime numbers then the prime to $p$ conductor of $A[p]$ is same as the conductor of $A$ and $B$.    In the Lemma \ref{compare} we have considered elliptic curves $A$ and $B$.  To prove Lemma \ref{compare} we have crucially used  \cite[Lemma 5.4]{MR}  and  \cite[Lemma 4.1.2] {EPW}. In fact,  \cite[Lemma 4.1.2] {EPW} is already stated for the Galois representation associated to general Hecke eigen forms and also one can easily prove \cite[Lemma 5.4]{MR} for abelian varieties.  In particular, the assertion of Theorem \ref{nontrivial}  holds for modular abelian varieties $A$ and $B$ also.   

For an elliptic curve $E$ defined over $\QQ$, let $L(E,s)$ be the Hasse-Weil $L$-function of $E$ defined over $\QQ$, $\Omega_E$ be the N\'eron period of $E$ over $\QQ$ and $R(E/\QQ)$ denotes the regulator of $E$ over $\QQ$. For a prime $p$, let 
\[\bar{\rho}_p(E) : \gal(\bar{\QQ}/\QQ) \lrta Aut_{\ZZ/p}(E[p]) \]
denote the residual Galois representation associated to $E$ at $p$. 

 We mention that it is well known   that for an elliptic curve $E/\QQ$ of rank zero   $\frac{L(E/\QQ)}{\Omega_E}$ is in $\QQ$. For an elliptic curve $E/\QQ$ of rank one a Theorem of Gross-Zagier implies that  $\frac{L'(E/\QQ)}{R(E/\QQ)\times \Omega_E}$ is in $\QQ$, where $L'(E,s)$ denotes the first derivative  $L(E,s)$.  For a detailed discussion on these results see \cite[Theorem 5.3]{JC}. In fact, a recent work of Christian Wuthrich implies that if   $E/\QQ$ is an elliptic curve with semistable reduction at an odd prime $p$  such that $\bar{\rho}_p(E)$ is an irreducible Galois representation then $\frac{L(E/\QQ)}{\Omega_E}$ is a $p$-adic integer if $L(E,1)\neq 0$ and $\frac{L'(E/\QQ)}{R(E/\QQ)\times \Omega_E}$ is a $p$-adic integer if $L(E,s)$ has a simple zero at $1$ (see \cite{Wu}).    We recall the following well known theorem (see for example \cite[Theorem 9.1]{SW}).

\begin{theorem}(Kato, Perrin-Riou)\label{zero}
Let $E/\QQ$ be an elliptic curve with $L(E,1)\neq 0$. Then ${\cyr \Sha}(E/\QQ)$ and $E(\QQ)$ are  finite and 
\begin{equation}\label{div}
  \# {\cyr \Sha}(E/\QQ)  \   |   \   C\times \frac{L(E,1)}{\Omega_E} \times \frac{\#E(\QQ))^2}{\prod c_{v,E}(\QQ)} 
\end{equation}
where $C$ is a product of a power of $2$, power of primes of additive reduction and power of primes for which the image of the Galois representation
\[  \bar{\rho}_p(E) :   \gal(\bar{\QQ}/\QQ) \lrta Aut_{\ZZ/p}(E[p])  \]
is not surjective. 
\end{theorem}

Theorem \ref{zero} says that  cardinality of $p$ primary part of ${\cyr \Sha}(E/Q)$ divides the quantity  in the right hand side of the relation \eqref{div} for every prime $p$ of good ordinary reduction, multiplicative reduction and good supersingular reduction as long as the image $\bar{\rho}_p(E)$ is surjective and $p$ is odd. The case of ordinary primes are due Kato ( \cite{KK}) and the case of Supersingular primes are due to Perrin-Riou (see \cite{PR}, see also \cite[Section 8]{SW} for more details).
We also mention that when $E$ has good ordinary reduction at the odd prime $p$ and does not admit complex multiplication then due to the work of Kato (\cite{KK}) in combination with the work of Skinner and Urban (\cite{SU}) on the main conjecture of elliptic curves with good ordinary reduction at $p$, the following much stronger result holds (see  \cite[Theorem 6.2]{JC}). 

\begin{theorem}\label{zero1}(Kato, Skinner-Urban)
Let $E/\QQ$ be an elliptic curve  with good and ordinary reduction at $p$. Suppose that \\
(i) $E$ does not admit complex multiplication,\\
(ii) $ \bar{\rho}_p(E) $ is  surjective,\\
(iii) there exists a prime $q$, where $E$ has bad multiplicative reduction, such that $\bar{\rho}_p(E)$ is ramified at $q$\\
(iv) $L(E,1)\neq 0$. . \\
Then the $p$ part of the Birch-Swinnerton-Dyer conjecture holds for $E$ , that is , ${\cyr \Sha}(E/\QQ)(p)$ and  $E(\QQ)$ are finite  and 
\[ ord_p({\cyr \Sha}(E/\QQ)(p))= ord_p(\frac{L(E,1)}{\Omega_E}) + 2ord_p(\#E(\QQ))-ord_p(\prod_v c_{v,E}(\QQ)) \]
\end{theorem}
Let $E$ be an elliptic curve with good  reduction at $p$ such that $L(E,1)\neq 0$. The $p$-part of  Birch-Swinnerton-Dyer conjecture also  holds for $E$ if it admits complex multiplication by an order in an imaginary quadratic field not containing the group of $p$-th root of unity (see \cite{R} or \cite[Theorem 6.1]{JC}).  If $L(E,s)$ has a simple zero at $s=1$ then a similar result holds for curves which admits complex multiplication without any condition on  the group of roots of unity (see \cite[Theorem 6.3]{JC}).  We do not have an analogue of Theorem \ref{zero1} in the case when $E$ has supersingular reduction at $p$ and it does not admit complex multiplication.  Nevertheless, using Theorem \ref{zero} we can prove the following corollary for a pair of elliptic elliptic curves $A$ and $B$ such that $A[p]\cong B[p]$ with good reduction  at $p$ (either ordinary or supersingular reduction).
\begin{corollary}\label{analy}
Suppose that $A$ and $B$ elliptic curves defined over  $\QQ$ with good reduction at  $p$  such that $A[p]\cong B[p]$ as $\gal(\bar{\QQ}/\QQ)$-modules.  Further we assume  that \\
(a) if $v$ is prime dividing $\n_A/\bar{\n}$ (resp. $\n_B/\bar{\n}$) then $A$ (resp. $B$) has non-split multiplicative reduction at $v$. \\
(b)  $A$ and $B$ are semi-stable over $\QQ$ .\\
(c)  The Mordell-Weil rank of $B$ is greater then the Mordell-Weil rank of $A$. \\
(d) $ \bar{\rho}_p(A)$ is surjective. \\
(e) $L(A,1)\neq 0$. \\
Then $p|\frac{L(A,1)}{\Omega_A}$. 
\end{corollary}
\begin{proof}
First we note that the assumption that $ \bar{\rho}_p(A)$ is surjective implies that $A[p]$ is an irreducible $\gal(\bar{\QQ}/\QQ)$-module and in particular $A(\QQ)[p]=0$. From Theorem \ref{nontrivial1} we get that $p|\#{\cyr \Sha}(E/\QQ)$. Also by an argument similar to Theorem \ref{nontrivial1}, we have that $\prod c_{v,A}(\QQ)$ is a $p$-adic unit. Surjectivity of $ \bar{\rho}_p(A)$ implies that the quantity $C$ in Theorem \ref{zero} is a $p$-adic unit. Therefore we get that $p|\frac{L(A,1)}{\Omega_A}$.
\end{proof}

We need the following theorem for elliptic curves of rank one over $\QQ$ (see \cite[Theorem 9.1]{SW}). For the definition of $p$-adic regulator of an elliptic curve $E$ with good and ordinary reduction at $p$, we refer the reader to \cite[Page 15]{SW} (see also \cite[Theorem 10.3]{Zh}). 

\begin{theorem}(Kato, Perrin-Riou)\label{one}
Let $E/\QQ$ be an elliptic curve with good and ordinary reduction at the odd prime $p$. Assume that the $p$-adic regulator  of $E$ is non-zero. Suppose that the representation 
\[   \rho_p(E)  : \gal(\bar{\QQ}/\QQ) \lrta Aut_{\ZZ_p}(T_pE)) \]
 is surjective where $T_pE$ denotes the Tate module of $E$ over $\QQ$. If $L(E,s)$ has a simple zero at $s=1$, then 
 \[  ord_p(\#{\cyr \Sha}(E/\QQ)(p)) \leq  ord_p\Big (\frac{(\#E(\QQ)_{tor})^2}{\prod_v c_{v,E}(\QQ)} \times \frac{1}{R(E/\QQ)}\times \frac{L'(E,1)}{\Omega_E} \Big) \]
 where  $R(E/\QQ)$ denotes the real valued regulator of $E$ over $\QQ$. 
\end{theorem}

Using Theorem \ref{one} and Theorem \ref{nontrivial1}, by an argument similar to the proof of Corollary  \ref{analy} we have the following corollary. 

\begin{corollary}\label{analy1}
Suppose that $A$ and $B$ are elliptic curves defined over  $\QQ$ with good and ordinary reduction at  $p$  such that $A[p]\cong B[p]$ as $\gal(\bar{\QQ}/\QQ)$-modules.  Further we assume  that \\
(a) if $v$ is prime  dividing $\n_A/\bar{\n}$ (resp. $\n_B/\bar{\n}$) then $A$ (resp. $B$) has non-split multiplicative reduction at $v$. \\
(b)  $A$ and $B$ are semi-stable over $\QQ$ .\\
(c)  The Mordell-Weil rank of $B$ is greater then the Mordell-Weil rank of $A$. \\
(d) $ \rho_p(A)$ is surjective. \\
(e) $L(A,s)$ has a simple zero at $s=1$. \\
Then $p | \frac{L'(A,1)}{R(E/\QQ)\times \Omega_A}$. 
\end{corollary}
In the context of assumption (d) of the above corollary we note that if $p\geq 5$ then $\rho_p(A)$ is surjective if and only if $\bar{\rho}_p(A)$ is surjective (see for example \cite[Proposition 7.2]{SW}). 

Given an elliptic curve $E$ defined over a number field $K$, let $d_p(E,K)$ be the smallest number such that there exists an extension of $K$ of degree  $d_p(E/K)$  and ${\cyr \Sha}(E/K)$ contains an element of order $p$. The integer $d_p(E/K)$ has extensively been studied by several mathematicians (see for example \cite{K}, \cite{C}, \cite{CS}).  In particular, it is known by the work of Clark and Shahed (\cite{CS}) that  $d_p(E/K)\leq p$.  It is generally expected that $d_p(E/K)\leq 2$.  It is an open problem to determine in general  if $d_p(E/K)$ is smaller then $p$. Indeed, in the above example we   see that $d_5(E_1/\QQ)=2$.  

\begin{remark}\label{remark}
Let $A$ be an abelian variety defined over $\QQ$. Suppose that there exist  abelian varieties $B$ and $J$ defined over $\QQ$ such that the following hold  \\
(a) There exists a quadratic extension $M/\QQ$ such that the rank of $A^\chi$ is less than the rank of $B$. \\
(b) $A[p]$ is irreducible as $\gal(\bar{M}/M)$ module. \\
(c) $J$ contains $A$ and $B$ as sub varieties such that Assumption \ref{assumption}  and also Assumption $(i)$ to $(iii)$ of Theorem \ref{quadratic} hold for $A$, $J$ and $B$ over $M$.\\
Then from Theorem \ref{quadratic} we get that ${\cyr \Sha}(A/M)$ contains an element of order $p$. 
\end{remark}

We mention that if $A[p]$ is an irreducible $\gal(\bar{\QQ}/\QQ)$-module and $\QQ(A[p])\cap M = \QQ$ then $A[p]$ is also an irreducible $\gal(\bar{M}/M)$-module. In particular, the condition $(b)$  holds for all but finitely many quadratic extensions $M/\QQ$.  Also, note that the above pair of elliptic curves with conductor $52$ and $264$ satisfy the assumptions $(a)$, $(b)$ and $(c)$ for $p=5$ and extension $M=\QQ(\sqrt{59})$. Another such example is the following pair of elliptic curves for $p=3$.  

\begin{example}
{ \normalfont
Let $A$ and $B$ be elliptic curves defined over $\QQ$ by the equations

\[   A :  y^2 + xy + y = x^3 - x^2 - 57x + 222   \]
\[   B :  y^2 + xy + y = x^3 - x^2 - 91x - 310   .\]
The conductor of $A$ is  $493$ and the conductor of $B$ is $17$. We have that $A[3]\cong B[3]$ as $\gal(\bar{\QQ}/\QQ)$-modules. Consider $J=(A\oplus B)/A[3]$ where $A[3]$ is considered as the subgroup scheme of $A\oplus B$ via the diagonal embedding. Then $A$, $B$ and $J$ satisfy condition $(a)$, $(b)$ and $(c)$ for $M=\QQ(\sqrt{195})$. If we denote the nontrivial  quadratic character of $\gal(M/\QQ)$ by $\chi$, then the rank of $A^\chi$ is $0$ and the rank of $B^\chi$ is $2$. In particular, the Shafarevich-Tate group of $A$ has a subgroup of order $9$. We mention that the elliptic curves $A^\chi$ and $B^\chi$ have bad reduction at the prime $3$. By a similar method, we can find a subgroup of order $9$ in the Shafarevich-Tate group of the 
curve $B$ over $M=\QQ(\sqrt{253})$. 
}
\end{example}
\qed

Let $\chi$ be a quadratic character (non-trivial) of the absolute Galois group of $\QQ$ with conductor $N_\chi$ and $M$ be the corresponding quadratic extension of $\QQ$. 
\begin{lemma}\label{twist}
Let $A$ be an elliptic curve defined over $\QQ$ with good reduction at the odd prime $p$. Suppose that  \\
(i) $\bar{\n}_A$ is coprime to $\n_A/\bar{\n}_A$ and $\n_A/\bar{\n}_A$ is square free, \\
(ii) a prime divisor $q$ of $\n_A/\bar{\n}_A$ is non split prime in $M$ if and only if $A$ has split multiplicative reduction at $q$, \\
(iii) $N_\chi$ is coprime to $p\n_A$.\\  
Then,  we have that the Tamagawa number of $A^\chi$ is a $p$-adic unit at every prime of $\QQ$ if one of the following holds \\
(a) $p > 3$,\\
(b) $p=3$ and $N_\chi$ is an odd number (i.e. $N_\chi\equiv 1$  mod $4$).  
\end{lemma}
\begin{proof}
Let $N_\chi$ denote the conductor of $\chi$ and $M$ be the quadratic field extension of $\QQ$ associated to $\chi$. From the assumption $(i)$ of the lemma, the conductor of $A$ is same as the conductor of $A[p]$ at primes dividing $\n_A$.   Since $N_\chi$ is coprime  to $N_A$,  at a prime $q|\bar{\n}_A$ the conductor of $A$ is same as the conductor $\A^\chi$ and the conductor of $A[p]$ at $q$ is same as the conductor of $A^\chi[p]$. Thus from Lemma \ref{compare}  we get that $c_{q,A}=c_{q,A}(\QQ)$ is a $p$-adic unit. Now we consider the case when $q|N_\chi$. In this case  $A^\chi$ has additive reduction at $q$. Therefore $c_{q,A^\chi}\leq 4$. In particular, if $p>3$  then we have that $c_{q,A}$ is a $p$-adic unit. When $p=3$, using the assumption that  $N_\chi$ is odd it can be shown that $c_{q,A^\chi}$ is a power of $2$ (see for example \cite[Section 3]{CW}). 

Finally we consider the case when $q|\n_A/\bar{\n}_A$. Since $\n_A/\bar{\n}_A$ is square free and $N_\chi$ is coprime to $\n_A/\bar{\n}_A$, $A^\chi$ has multiplicative reduction at $q$. The Tamagawa number of $A^\chi$ at  $q$ is a $p$-adic unit if  $A^\chi$ has non-split mulplicative reduction at $q$.  Let $f_A$  denotes the modular form associated to $A$ via modularity and $a_q(f_A)$ denotes the $q$-th Fourier coefficient of $f_A$ . Then $A$  has non-split multiplicative  reduction at $q$ if and only if $a_q(f_A)$  is equal to $-1$ and $A$ has split multiplicative reduction at $q$ if the $a_q(f_A)$  is equal to $1$ at $q$. A similar assertion holds for $A^\chi$. From the assumption $(ii)$ of the lemma, we get that $a(f_A)=a(f_{A^\chi})$ if $A$ has non-split multiplicative reduction at $q$ and $a(f_A)= -a(f_{A^\chi})$ if $A$ has split multiplicative reduction at $q$. In particular, we get that $a(f_{A^\chi})=-1$ for every $q| \n_A/\bar{\n}_A$. This implies that $A^\chi$ has non-split multiplicative reduction at prime divisors of $\n_A/\bar{\n}_A$ and therefore $c_{q,A^\chi}$ is a $p$-adic unit at every prime $q| \n_A/\bar{\n}_A$. 
\end{proof}

Let $A$ and $B$ are elliptic curves defined over $\QQ$ such that $A^\chi$ and $B^\chi$ satisfy conditions of Lemma \ref{twist} for a quadratic character associated to a quadratic extension $M$. Then from Lemma \ref{twist} and Remark \ref{remark} we get that the Shafarevich Tate group of $A$ has an element of order $p$ over $M$  if $A$ and $B$ satisfy conditions $(a)$ and $(b)$ of remark \ref{remark}. 

We also mention  that, in the case when $p=3$, $N_\chi\equiv 1$ mod $4$ is a sufficient but not necessary condition. 

\begin{example}
{ \normalfont
We have the following pair of elliptic curves with conductor $203$

\[ E_1 :  y^2 + y = x^3 - x^2 + 20x - 8   \]

\[  E_2 :  y^2 + xy = x^3 + x^2 - 9x + 8  \]
satisfying $E_1[3]\cong E_2[3]$. The Mordell-Weil rank of $E_1$ and $E_2$ are zero. The Tamagawa numbers for $E_1^{(3)}$ and $E_2^{(3)}$ are $3$ adic units at every prime of bad reduction. Note that the conductor of the quadratic character of the quadratic extension $\QQ(\sqrt{3})/\QQ$ is $12$.  The rank of $E^{(3)}_1$ is zero and the rank of $E_2^{(3)}$ is $2$ over $\QQ$. In particular this implies ${\cyr \Sha}(E_1/\QQ(\sqrt{3}))$ has a visible subgroup of order  $9$.  Similarly we can show that ${\cyr \Sha}(E_1/\QQ(\sqrt{23})$ has a visible subgroup of order $9$.
}
\end{example}
\qed

Given an elliptic curve $E$ defined over $\QQ$, by a result of Rubin and Silverberg (see \cite{RS}) there exists infinitely many elliptic curves defined over $\QQ$ with residual Galois representation  equivalent to the residual Galois representation  of $E$ at  $p=3$ and $p=5$. One can also find similar examples for $p=7$ and $11$ (see \cite{F}).  For larger primes such examples of pair of elliptic curves with equivalent residual Galois representations are difficult to find. But by the level raising result of modular forms, one can find  examples of modular abelian varieties with residual representation equivalent to the residual representation of  given elliptic curve $E$ at larger primes $p$ (see \cite{D}).

\section{visibility over $p$-extensions}\label{prime1}
As in the previous section, we fix a Galois extension $L$ of $\QQ$. Let $M/L$ and $K/L$ be Galois extensions of $L$ such that $K\subset M$. Fix an odd prime $p$. 
In the next theorem for a number field extension $M/K$, 
for simplicity we denote the map $\tau_{K}^{M} : \Vis_J(K,A) \lrta \Vis_J(M,A)$ by $\tau$. 
\begin{theorem}\label{exten}
Suppose that Assumption \ref{assumption} holds for $n=p$. Further assume that \\
(i) the ramification index  $e_v$ of  $M$ at primes dividing $N/N_A$ is divisible by $p$,\\
(ii) $p\nmid \prod_{v|N_A(K)} c_{v,A}(K)c_{v,B}(K).$\\ 
 Then the composition 
\[   B(K)/pB(K)  \stackrel{\phi}{\lrta} Vis_J(K,A) \stackrel{\tau}{\lrta} Vis_J(M,A) \]
factors through a natural map $B(K)/pB(K) \lrta Vis_J({\cyr \Sha}(M/A))$. 
We have that $\tau\phi$ is an injective homomorphism if the rank of $A(M)$ is zero. More generally, if $M$ is a cyclic extension of $K$ of degree $p$  then the $\ZZ/p\ZZ$-rank of the kernel of $\tau\phi$ is bounded by $(rank(A(K))+\frac{rank(A(M))-rank(A(K))}{p-1})$. \\
\end{theorem}
\begin{proof}
For a prime $v$ of $K$, let $\kappa_w$ denote the restriction map
\[  H^1(K,A)  \lrta H^1(K_v,A)  \]
and for a prime $w|v$, let $\tau_w$ be the restriction map $H^1(K_v,A) \lrta H^1(M_w,A)$.  Let $x\in B(K)$. To show that $\tau\phi(x)\in Vis_J({\cyr \Sha}(M/A))$, it is enough to show that image of $\tau\phi$ lies in the kernel of $\tau_w\kappa_v$ for every prime $v$ of $K$ and primes $w|v$ of $M$. First, we suppose that $v$ is an Archimedean prime. Then the order of image of $x$ in $H^1(K_v,A)$  divides  $2$. On the other hand from Lemma \ref{vis} we get that the order of the image of $x$ divides the odd prime $p$. This implies that the image of $x$ in $H^1(K_v,A)$ is zero. Therefore $\tau_w\kappa_v\phi(x)=0$. 

Next we suppose that $v|N_A(K)$. From the assumption of the theorem, we have that $c_{v,A}(K)c_{v,B}(K)$ is a $p$ adic unit. Therefore  it follows from the  proof of \cite[Theorem 3.1]{AS} that the image of $x$ in $H^1(K_v,A)$ is zero. This implies that the image of $x$ in $H^1(M_w,A)$ is zero. Now we consider the case when $v|N(K)/N_A(K)$. In particular, $A$ has good reduction at $v$. Thus it follows from \cite[Section 4, Corollary 1]{LT} that  the image of $x$ in $H^1(M_w,A)$ is trivial. If $w\nmid pN$ then the Tamagawa number of both $A$ and $B$ are $p$-adic unit over $K$ and it follows form \cite[Theorem 3.1]{AS} that the image of $x$ in $H^1(K_v,A)$ is trivial. 

Finally, we consider the case when $v|p$. Let $C$ be the abelian variety defined as $J/A$ over $L$. Since ramification $e_p$  of $L$ at primes dividing $p$ satisfies $e_p<p-1$, $A$ and $B$ are defined over $L$, from Lemma \ref{smooth1} we get that the sequence 
\[  0 \lrta  A(K_w^{ur}) \lrta J(K_w^{ur}) \lrta C(K_w^{ur}) \lrta 0  \]
is exact. Now it follows from an argument similar to the proof of \cite[Theorem 3.1]{AS} that the image $x$ in $H^1(K_v,A)$ is zero. Therefore the image of $x$ in $H^1(M_w,A)$ is also zero. This proves the first part of the theorem. 

Suppose that $M/K$ is a cyclic extension of degree $p$. Next, we use the $G:=\gal(M/K)$-module structure of $A(M)$ to determine a bound on the size of the kernel of $\tau\phi$. 
By the structure of module over the group ring $\ZZ_p[G]$-we get that 
 \[  A(M)\otimes \ZZ_p \cong \ZZ_p^a \oplus \ZZ_p[G]^b \oplus I(G)^c \] 
 for some positive integers $a$, $b$ and $c$ where $I(G)$ denotes the augmentation ideal  of the group ring $\ZZ_p[G]$.  Now using the facts that $H^1(G,\ZZ_p)=0$, $H^1(G,\ZZ_p[G])=0$ and $H^1(G,I(G))=\ZZ/p\ZZ$, we get that $H^1(G,A(M)\otimes \ZZ_p) \cong (\ZZ/p\ZZ)^c$. Since $\ZZ_p$ is a flat $\ZZ$-module, we have that $H^i(G,X)\otimes \ZZ_p\cong H^i(G,X\otimes \ZZ_p)$ for every $i$ and $\ZZ[G]$-module $X$.   For $i\geq 1$, $H^i(G,A(M)$ is a finite $p$-torsion abelian group and therefore $H^i(G,A(M))=H^i(G,A(M))\otimes \ZZ_p=H^i(G,A(M)\otimes \ZZ_p)$. Also note that the rank of $\ZZ_p$-module $A(K)\otimes \ZZ_p$ is equal to the $\ZZ_p$ rank of $H^0(G,A(M)\otimes \ZZ_p)$ which is same as $a+b$. 
 
From snake lemma we get the exact sequence,
\[  0 \lrta ker(\phi) \lrta ker(\tau\phi) \lrta ker(\tau)  . \]
From \cite[Lemma 3.7]{AS}  the $\ZZ/p\ZZ$-rank of  $ker(\phi)$ is bounded by the $\ZZ$-rank $a+b$ of $A(K)$ which is same as the $\ZZ_p$ rank of $A(K)\otimes \ZZ_p$.    Since $ker(\tau)$ is a subgroup of $H^1(G,A(M))$, the $\ZZ/p\ZZ$-rank of $\ker(\tau)$ is bounded by $c$. This implies that the $\ZZ/p\ZZ$-rank of $ker(\tau\phi)$ is bounded by $a+b+c\leq rank(A(K))+\frac{rank(A(M))-rank(A(K))}{p-1}$. 

Note that if $\gal(M/K)$ is a group of order prime to $p$ then $H^1(G,A(M)[p]=0$. This implies that the $ker(\tau)[p]=0$. Therefore the kernel of $\tau\phi$ is bounded by the kernel of $\phi$. Therefore in this case the $\ZZ/p\ZZ$-rank of the kernel of $\tau\phi$ is bounded  the rank of $A(K)$.  
 \end{proof}
Next  using an example we illustrate how the above theorem can be used to produce visible elements in the  Shafarevich-Tate groups of abelian varieties. 
\begin{example}
{ \normalfont
Let $E_1$ and $E_2$ be elliptic curves of conductor $176$ and $1232$ respectively defined by the equations
\[  E_1  :   y^2 = x^3 + x^2 - 5x - 13 \]  
and 
\[ E_2   :  y^2 = x^3 + x^2 + 56x - 588   \]
We have that $E_1[3]\cong E_2[3]$ as $\gal(\bar{\QQ}/\QQ)$-modules. Tamagawa numbers of elliptic curves $E_1$ (resp. $E_2$) at primes dividing $176$ are $3$-adic units over $K=\QQ(\mu_3)$. The curve $E_1$ has good reduction at the primes of $K$ 
dividing $7$ and $E_2$ has split multiplicative reduction at primes dividing $7$. 
The Tamagawa number of $E_2$ at the primes dividing $7$ of $K$ are not $3$-adic units.  

As in previous section, put $J=(E_1\oplus E_2)/E[3]$ where we view  $E[3]$ as a subvariety of $E_1\oplus E_2$ via the diagonal embedding and let $E_i$ maps to $J$ via the imbedding 
\[  E_i \hookrightarrow E_1\oplus E_2  \twoheadrightarrow (E_1\oplus E_2)/E_1[3]  \]  
 for $i=1,2$.  Note that under this map $E_1[3]$ maps to the image of $E_2[3]$ and $E_1\cap E_2=E_1[3]=E_2[3]$ in $J$.  Further it can be checked that $E_1[3]\cong E_2[3]$ are irreducible $\gal(\bar{\QQ}/K)$-modules.   In particular, this implies that $E_1(K)[p]=0$. Since $J/E_1$ is isogenous to $E_2$ and $E_2[3]$ is an irreducible $\gal(\bar{\QQ}/K)$-module, $(J/E_1)(K)[p]=0$.  We consider the field extension $M=K(7^{1/3})$ over $K$. Since $M/K$ is ramified at $7$,  we get a map from
\[  \tau^M_K\phi_K : E_2(K) \lrta  Vis_J({\cyr \Sha}(E_1/M))    \]
as in Theorem \ref{exten}. Now, it can be checked using Sage that the rank of $E_1$ over $M$ is zero and the rank of $E_2$ over $K$ is two. Therefore $\tau^M_K\phi_K$ is an injective map and its image is a rank two visible subgroup in  $Vis_J({\cyr \Sha}(E_1/K))\subset {\cyr \Sha}(E_1/K)$.
}
\end{example}
\qed

Let $A$ and $B$ be abelian varieties defined over $\QQ$ such that $A[p]\cong B[p]$ as group schemes. Suppose that $A[p]$ is irreducible as $\gal(\bar{\QQ}/\QQ(\mu_p))$-module.  Put $J:=(A\oplus B)/A[p]$ where $A[p]$ is viewed as the subgroup scheme of $A\oplus B$ via the diagonal embedding $A[p]\hookrightarrow A\oplus B$.  Put $K=\QQ(\mu_p)$ and $M=K(m^{1/p})$ where $m=N(K)/N_A(K)=lcm(N_A(K),N_B(K))/N_A(K)$. Since $A[p]\cong B[p]$ is irreducible as $\gal(\bar{\QQ}/\QQ)$-module, from Remark \ref{const} we have that  $A(K)[p]=0=(J/A)(K)[p]$. Note that  $M/K$ is a degree $p$-extension and this implies that  $A(M)[p]=0$.  We suppose that $\prod_{v|N_A}c_{v,A}(K)c_{v,B}(K)$ is a $p$-adic unit.  Also, note that the ramification index of $M/K$ at primes dividing $m$ is equal to $p$. Therefore from Theorem \ref{exten} we have a natural homomorphism 
\[  B(K)/pB(K) \lrta  Vis_J({\cyr \Sha}(M/A)) \]
We consider a family of abelian varieties $\{B_t\}_{t\in I}$ over $\QQ$ satisfying the above conditions and  $B_t[p]\cong A[p]$ for all $t$. Let $N_t$ denote the product of primes of bad reduction of $A$ or $B$ over  and put $m_t=N_t(K)/N_A(K)$, $K=\QQ(\mu_p)$ and $M_t=K(m_t^{1/p})$.  Then the rank of $B_t$ over $K=\QQ(\mu_p)$ is bounded by  $(rank(A(M_t))-rank(A(K)))/(p-1)+c$ where $c$ denotes the $\ZZ/p$-rank of ${\cyr \Sha}(A/M_t)[p]$.  Thus to study the variation of ranks of abelian varieties in the given family over $K$, we need to understand the variation of the rank and cardinality of the $p$-torsion subgroup  of the Shafarevich-Tate group and Mordell-Weil rank of the fixed abelian variety over Kummer extensions of the form of $K(m^{1/p})$ as $m$ varies over positive integers. 

As an example,  consider the elliptic curve A=11A3 of conductor $11$ from Cremona table defined by the  equation
\[  A :  y^2  +  y =  x^3 - x \]
and $p=3$. The Tamagawa number of $A$ over $K=\QQ(\mu_3)$ is $3$-adic unit at the primes of $K$ dividing $N_A(K)=11$. Also the residual representation $A[3]$ of $A$ at $3$ is irreducible as $\gal(\bar{K}/K)$-module. The rank of $A$ over $K$ is trivial. From \cite{RS} we get an infinite family of elliptic curves $\{B^t\}_{t\in \ZZ}$ defined over $\QQ$ such that $A[p]\cong B^t[p]$ as $\gal(\bar{\QQ}/\QQ)$-module. If the Tamagawa number of $B^t$ is a $3$-adic unit at the primes dividing $11$ of $K$,  we get the natural map 
\[  B(K)/pB(K) \lrta  Vis_J({\cyr \Sha}(M/A)) \]
as in Theorem \ref{exten}. 
Since the rank of $A$ is trivial over $K$, the kernel of $\tau\phi$ is a subgroup of $H^1(K(n^{1/p})/K, A(K(n^{1/p})))$ where $n=N(K)/N_A(K)$.
 
We shall end the section with the following result describing the visibility map over certain $p$-adic Lie extensions of a number field.  
Let $K/L$ be a finite Galois extension of number fields and $K_\infty$ be a $p$-adic Lie pro $p$-extension of $K$. Let $G$ denote the Galois group of $K_\infty/K$. For a prime $w$ of $K_\infty$, let $G_w$ be the decomposition subgroup of $G$ associated to $w$. For an abelian variety $A$ defined over $L$, recall that  the Shafarevich-Tate group ${\cyr \Sha}(A/K_\infty)$ of $A$ over $\K$ is defined   as
\[ {\cyr \Sha}(A/K_\infty):= \ilim_{K} {\cyr \Sha}(A/K) \]   
where $K$ varies over the set of number fields $K\subset K_\infty$ containing $L$. The direct limit is taken with respect to the natural restriction maps of Galois cohomology groups.   The visible subgroup of ${\cyr \Sha}(A/\K)$ with respect to $i$ is defined as 
\[  \Vis_J({\cyr \Sha}(A/\K)) := {\cyr \Sha}(A/\K)\cap \Vis_J(\K,A) = ker ({\cyr \Sha}(A/\K)\lrta {\cyr \Sha}(J/\K) )\]

\begin{lemma}\label{vis1}
Let $A$ and $B$ be abelian subvarieties of an abelian variety $J$ defined over $L$ such that $A\cap B$ is finite and $B[n] \subset A\cap B$. 
Let $K$ be a finite extension of $L$ such that 
$(J/B)(K)[p]=0$ and $B(K)[p]=0$. Then 
there is a natural map
\[   \phi_{\K}  :  B(\K) / pB(\K)  \lrta \Vis_J(\K,A)  \]
such that $ker(\phi) \subset J(\K)/(B(\K) + A(\K) ) $. If $A(\K)$ is finite then $\phi$ is an injective map. 
\end{lemma}
\begin{proof}
Since $A(K)[p]=0$ and $\K/K$ is pro-$p$ extension of $K$, by Nakayama Lemma we get that $A(\K)[p]=0$. Similarly $(J/C)(\K)[p]=0$. In particular for every extension $F$ of $K$ such that $F\subset \K$, from Lemma \ref{vis} we have a natural map
\[     B(F) / pB(F)  \stackrel{\phi_F}{\lrta} \Vis_J(F,A)  \]
which is injective if $A(F)=0$.  
Now the lemma follows taking the direct limit of $\phi_F$ over natural restriction maps of Galois cohomology groups as $F$ varies over Galois extension $F\subset \K$ of $K$. 
\end{proof}
  
\begin{theorem}\label{Lie}
Let $A$ and $B$ be abelian sub varieties of an abelian variety $J$ defined over $L$ such that Assumption \ref{assumption} holds for $n=p$.  Let  $v|N$ be a prime of $K$ such that  the Tamagawa number of either $A$ or $B$ is not a $p$-adic unit at $v$ over a finite extension of $K$ in $\K$.  Assume that the dimension of $G_w$ as a $p$-adic Lie group is $\geq 2$ for every prime $w|v$ of $\K$ .\\
Then the natural map $\phi_{\K} : B(K_\infty)/pB(K_\infty) \lrta  \Vis_J(K_\infty,A)$ factors through 
\[  \phi_{\K} : B(K_\infty)/pB(K_\infty)  \lrta \Vis_J{\cyr \Sha}(A/K_\infty) \]
such that $ker(\phi)\subset J(K_\infty)/(B(K_\infty)+A(K_\infty))$. In particular, $\phi_{\K}$ is injective if $A(\K)$ is a finite group. 
\end{theorem}
\begin{proof}
For a prime $w$ of $K_\infty$, let $ \kappa_w$ be the natural restriction map 
\[  H^1(K_\infty,A) \lrta H^1(K_{\infty,w},A).  \]
Here $K_{\infty,w}$ denotes the union of $K'_w$ in a fixed algebraic closure $\bar{K_w}$ of $K_w$, $K'$ varies over the finite extensions of $K$ in $K_\infty$ and $K'_w$ denotes the completion of $K'$ at  $w$. 
Let $x\in B(K_\infty)$.  To show that $\phi(x)\in Vis_J{\cyr \Sha}(A/K_\infty)$, it is enough to show that $\kappa_w(\phi(x))=0$ for every prime $w$ of $K_\infty$.  

We first consider the case when $w$ is an Archimedean  prime. Note that from Lemma \ref{vis} we have that $p$ is divisible by the order of $\phi(x)$. Since $w$ is an Archimedean prime, we get that $\kappa_w(\phi(x))$ is annihilated by $2$. Since $p$ is an odd prime, we concluded that $\kappa_w(\phi(x))=0$. Let $v$ be a finite prime of $K$. Suppose that the Tamagawa number of either $A$ or $B$ is not a $p$-adic unit over an extension of $K$ in $\K$. Then from the assumption of the theorem, for every prime $w$ of $\K$, the dimension of $G_w$ is at least $2$.   In this case, it follows from \cite[Lemma 5.4(ii)]{OV} and assumption (d)  that $H^1(K_{\infty,w},E[p])=0$. From Kummer theory for abelian varieties,  we have a surjective homomorphism from $H^1(K_{\infty,w},E[p])\lrta H^1(K_{\infty,w},A)[p]$. This implies that  $H^1(K_{\infty,w},A)[p]=0$. Since $p$ is divisible by the order of $\phi(x)$, we get that $\kappa_w(\phi(x))\in H^1(K_{\infty,w},A)[p]$ which implies  $\kappa_w(\phi(x))=0$.

Next we consider the case when the Tamagawa numbers of both $A$ and $B$ are $p$-adic units at $v$ over every extension of $K$ in $\K$. Choose a finite extension $M$ of $K$ in $\K$ such the $x$ is defined over $M$. Then from the proof of \cite[Theorem 3.1]{AS}, we get that the image of $x$ in $H^1(M_w,A)$ is zero. In particular, this implies that the image of $\phi(x)$ in $H^1(K_{\infty,w},A)$ is zero. 
\end{proof}

Let $v$ be a finite prime of $K$ and $F$ be a finite $p$-extension of $K$ and $w|v$ be a prime of $F$. Let $K^{ur}_v$ (resp $F^{ur}_w$) be the maximal unramified extension of $K_v$ (resp $F_w$).  Suppose that $F_w/K_v$ is unramified. We have a surjective map from $H^1(F_w^{ur}/K_w,A(F^{ur})) \lrta H^1(F_w^{ur}/F_w,A(F^{ur}))^{\gal(F_w/K_v)}$.  Suppose that the Tamagawa number of an abelian variety $A$ over $K_v$ is a $p$-adic unit. Then the cardinality of  the group $H^1(K_w^{ur}/K_w,A(K^{ur}))$ is a $p$-adic unit. Since $F_w/K_v$ is unramified, we have that $F_w^{ur}=K_v^{ur}$. This implies that the cardinality of  $H^1(F_w^{ur}/F_w,A(F^{ur}))^{\gal(F_w/K_v)}$ is a $p$-adic unit. Since ${\gal(F_w/K_v)}$ is a $p$-group, by Nakayama lemma we get that the cardinality of  $H^1(F_w^{ur}/F_w,A(F^{ur}))$ is a $p$-adic unit. Thus, we have that if $F_w/K_v$ is unramified $p$-extension and the Tamagawa number of $A$ over $K_v$ is a $p$-adic unit then  it is a $p$-adic unit over $F_w$ also. 

\begin{example}\label{ex5}
{ \normalfont
Let $E_1$ and $E_2$ be the pair of elliptic curves defined by the equations
\[  E_1 :  y^2=x^3+x^2 - 30008176 x - 63229110828 \]
and 
\[ E_2 :   y^2=x^3 + x^2 - 144 x + 532 \]
respectively over $\QQ$. The conductor of $E_1$ and $E_2$ over $\QQ$ is $3536=2^4\times 13 \times 17$ . We have that $E_1[3]\cong E_2[3]$ as $\gal(\bar{\QQ}/\QQ)$-modules and are irreducible.   Let $\Q_{cya}$ denote the cyclotomic $\ZZ_p$-extension of $\QQ$. Both $E_1$ and $E_2$ has good reduction at $p$.  The Tamagawa numbers of $E_1$ and $E_2$ are $3$-adic unit at every finite prime of $\QQ$. Since $\QQ_{cyc}/\QQ$ is unramified at every prime not dividing $3$, for every extension $F/\QQ$ in $\QQ_{cya}$, the Tamagawa number of $E_1$ and $E_2$ at every finite prime of $F$  is a $3$-adic unit. From Theorem \ref{Lie} we get that the natural map  
$\phi_{\QQ_{cyc}} : B(\QQ_{cyc})/pB(\QQ_{cyc}) \lrta  \Vis_J(\QQ_{cyc},A)$ factors through 
\[  \phi_{\K} : B(\QQ_{cyc})/pB(\QQ_{cyc})  \lrta \Vis_J{\cyr \Sha}(A/\QQ_{cyc}) .\]
In fact, in this case it can be shown that the Mordell-Weil Iwasawa $\lambda$-invariant of $E_1$ is zero and the Mordell-Weil $\lambda$-invariant of  $E_2$ is two two. This in particular, implies that   $E_1(\QQ_{cyc})$ is a finite group and $E_2(\QQ_{cyc})$ is a finite generated abelian group of rank $2$ and therefore $\phi_{\QQ_{cyc}}$ is non-zero injective map. We mention that to compute the Mordell-Weil Iwasawa Lambda invariant of $E_1$ and $E_2$ one need to assume that the Birch and Swinnerton Dyer conjecture for both $E_1$ and $E_2$ are valid over every extension of $\QQ$ in $\QQ_{cyc}$. For the definition and a large data base on Mordell-Weil Iwasawa $\lambda$-invariant of  an elliptic curve we refer the reader to the link \cite{P}. 
}
\end{example}

Now consider an abelian variety $J$ and  a pair of  abelian sub varieties   $A$ and $B$ of $J$ over $\QQ$ with good  reduction at primes dividing $p$. Suppose that $A$, $B$ and $J$ satisfy Assumption  \ref{assumption}. Put $K=\QQ(\mu_p)$ and let $m$ be the product of primes number $q$ such that the Tamagawa number of either $A$ or $B$ over $K$ at primes dividing $q$ is not a $p$-adic unit. Put $\K:= \cup_rK(\mu_{p^r}, (pm)^{1/p^r})$ where $r$ varies over the set of positive integers. The $p$-adic Lie extension $\K$ is unramified outside the primes of $K$ dividing $mp$ and the  dimension of the decomposition subgroup $G_w$ of $\gal(\K/K)$ at every prime $w$ of $\K$ dividing $mp$ is $2$. Thus $\K$ satisfies all the assumptions of Theorem \ref{Lie}. Therefore we get that  the natural map $\phi_{\K} : B(K_\infty)/pB(L_\infty) \lrta  \Vis_J(K_\infty,A)$ factors through 
\[  \phi_{\K} : B(K_\infty)/pB(K_\infty)  \lrta \Vis_J{\cyr \Sha}(A/K_\infty) \]
such that $ker(\phi)\subset J(K_\infty)/(B(K_\infty)+A(K_\infty))$. 
At present we do not know in general  how to determine if the map $\phi_{\K}$ is non-trivial. Using the techniques from Iwasawa theory we hope to answer this question in our future  work.

\end{document}